\documentclass[11pt]{amsart}
\usepackage{amsfonts}
\usepackage{amssymb}
\usepackage[arrow,matrix,curve]{xy}
\usepackage{ifthen}

\usepackage[active]{srcltx}

\newtheorem{theorem}{Theorem}[section]
\newtheorem{definition}[theorem]{Definition}

\newtheorem{lemma}[theorem]{Lemma}
\newtheorem{proposition}[theorem]{Proposition}
\newtheorem{corollary}[theorem]{Corollary}

\newcommand{\N}{{\mathbb N}}

\newcommand{\Z}{{\mathbb Z}}
\newcommand{\F}{{\mathbb F}}

\newcommand{\C}{{\mathbb C}}

\newcommand{\T}{{\mathbb T}}
\newcommand{\I}{{\mathbb I}}

\renewcommand{\P}{{\mathbb P}}

\def\cal{\mathcal}

\newcommand{\cala}{{\cal A}}
\newcommand{\calo}{{\cal O}}

\newcommand{\calp}{{\cal P}}

\newcommand{\calt}{{\cal T}}

\newcommand{\aut}[0]{ \mbox{\rm{Aut}}}

\newcommand{\set}[2]{\{ \, #1 \,\, | \, \, #2 \, \} }

\newcommand{\lin}[0]{ \mbox{\rm{span}}}

\begin{document}

\title{A Cuntz--Krieger uniqueness theorem for semigraph $C^*$-algebras}
\author[B. Burgstaller]{B. Burgstaller}
\address{Doppler Institute for mathematical physics, Trojanova 13, 12000 Prague, Czech Republic}
\email{bernhardburgstaller@yahoo.de}
%

\begin{abstract}
Higher rank semigraph algebras are introduced by mixing concepts of ultragraph algebras
and higher rank graph algebras.
This yields a kind of higher rank generalisation of ultragraph algebras.
We prove Cuntz--Krieger uniqueness theorems for cancelling semigraph algebras and
aperiodic full semigraph algebras.
%
\end{abstract}

\maketitle


\section{Introduction}

Tomforde's ultragraph algebras \cite{tomforde} and Bates and Pask's $C^*$-algebras of labelled graphs \cite{batespask}
are $C^*$-algebras which generalise graph algebras \cite{kumjianpaskraeburnrenault}
by introducing - beside a directed graph - a further projection set which allows higher flexibility to design the $C^*$-algebra.
For instance, Exel--Laca algebras \cite{exellaca} are ultragraph algebras according to Tomforde \cite{tomforde}, but are only proved to be Morita equivalent to graph algebras \cite{katsuremuhlysimstomforde}.
In another direction, graph algebras by Enomoto and Watatani \cite{enomotowatatani}
and Kumjian, Pask, Raeburn and Renault \cite{kumjianpaskraeburnrenault}
were generalised to higher rank graph algebras by Kumjian and Pask in \cite{kumjianpask}
and Raeburn, Sims and Yeend in \cite{raeburnsimsyeendFinitelyAligned}.
A central result for Cuntz--Krieger algebras \cite{cuntzkrieger}, ultragraph algebras and labelled graph $C^*$-algebras is the existence of a Cuntz--Krieger uniqueness theorem,
firstly proved for the Cuntz--algebras \cite{cuntz}.

In this work we extend Tomforde's concept \cite{tomforde} of allowing an extra projection set in the construction of the algebra
to higher rank graphs. Such a graph algebra will be called a higher rank semigraph algebra, see Definition \ref{DefSemigraphAlgebra}.
We do not use a strict concept by decorating the graph, but use a slightly more flexible concept
by allowing the algebra to be generated by partial isometries coming from a higher rank semigraph (Definition \ref{Defsemigraph})
and a projection set, and mix it with relations which are reminiscent of the relations of higher rank graph algebras \cite{kumjianpask}.
Then ultragraph algebras, $C^*$-algebras of labelled graphs and higher rank Exel--Laca algebras \cite{burgiHREL} are
higher rank semigraph algebras.
We prove a Cuntz--Krieger uniqueness theorem for cancelling semigraphs (Definition \ref{defCancel}) in Theorem \ref{theoremamenablesemigraphsystem}.

A side benefit of higher rank semigraph algebras
is that they are stable under quotients (provided the quotient allows a gauge action),
see Lemma \ref{lemmaQuotientSemigraphAlgebra}, and so are predestinated for studying quotients (see also \cite{burgievansProperties}).
In the theory of graph algebras one has to go over to relative graph algebras as studied by
Sims \cite{sims} when dealing with quotients.

In Section \ref{SectionFullSemiAlgebra}
we associate to every semigraph algebra another semigraph algebra, called the full semigraph algebra, by adding relations
which are analogs to Cuntz' relation
$s_1 s_1^* + s_2 s_2^*=1$ in the Cuntz algebra $\calo_2$.
The main result of this section is that an aperiodic full semigraph algebra (Definition \ref{definitionAperiodicity}) is cancelling, see Proposition \ref{fullaperiodicsystemamenable}, and so statisfies
the Cuntz--Krieger uniqueness theorem according to Theorem \ref{theoremamenablesemigraphsystem}. Our aperiodicity condition may be compared with Cuntz and Krieger's aperiodicity
condition in \cite{cuntzkrieger} or Lewin and
Sim's aperiodicity condition in \cite{simsAperiodicity} for higher rank graph algebras.

We give a brief overview of this paper.
In Sections \ref{sectionSemimultiSet}-\ref{SectionSemigraph} we introduce the notion of a finitely aligned $k$-semigraph.
In Sections \ref{SectionSemigraph}-\ref{SectionsSemigraphAlgebraCore} we define higher rank semigraph algebras and make sufficient
analysis (in particular of the core) to be prepared for the proof of the Cuntz--Krieger uniqueness theorem, Theorem \ref{theoremamenablesemigraphsystem}, for
cancelling semigraph algebras in Section \ref{SectionCuntzKriegerUniqueness}.
In Section \ref{SectionSemigraphQuotient} we state stability under quotients,
and in Section \ref{SectionFullSemiAlgebra} we discuss full semigraph algebras.

\section{Semimultiplicative sets}  \label{sectionSemimultiSet}

In higher rank graph $C^*$-algebra theory \cite{kumjianpask} a
graph is a small category.
We are going to introduce higher rank semigraph $C^*$-algebras
which are relying on a similar but more general structure
%
called a semimultiplicative set.

\begin{definition}  \label{definitionSemimultiSet}
{\rm
A {\em semimultiplicative set} $T$ is a set equipped with
a subset $T^{(2)} \subseteq T$ and a multiplication
$$T^{(2)} \longrightarrow T : (s,t) \mapsto st,$$
which is associative, that is, for all $s,t,u \in T$,
$(s t) u$ is defined if and only if $s (t u)$ is defined, and
both expressions are equal if they are defined.
}
\end{definition}

When we say $(s t) u$ is defined then we mean $(s,t) \in T^{(2)}$ and
$(st,u) \in T^{(2)}$.
An example which is relevant for us is the semimultiplicative set $\Lambda$ which is a graph \cite{kumjianpask}.
Then the product $\lambda \mu$ of two elements of $\Lambda$ is defined if and only if
$s(\lambda) = r(\mu)$. A graph is even a semi-groupoid \cite{exelSemigroupoid}. A semi-groupoid is a semimultiplicative set
with the property that $(s t) u$ is defined if and only if both $st$ and $tu$ are defined.
The second example - and this comes closer to what we do here in this paper -
is to think of a semimultiplicative set as a graph $\Lambda$ without the projections,
so the set $T=\Lambda \backslash \Lambda^{(0)}$.

\section{Semigraphs}   \label{SectionSemigraph}

We define $\N_0 = \{0,1,2,\ldots\}$ and denote by $\T$ the unit circle.
In this paper $k$ denotes an index set.
If $k$ is infinite then $\Z^k$
denotes the set of
functions $f:k \rightarrow \Z$ with {\em finite} support
(and similarly $\N_0^k$ and $\T^k$ denote the set of functions with finite support).

\begin{definition}  \label{Defsemigraph}
{\rm
Let $k$ be an index set (which may be regarded as a natural number if $k$ is finite).
A {\em $k$-semigraph} $T$ is a semimultiplicative set $T$ equipped with a map
$d:T \longrightarrow \N_0^k$ satisfying the {\em unique
factorisation property} which consists of the following two conditions:

(1) For all $x,y \in T$ for which the product $xy$ is defined one has
$$d(xy)= d(x)+d(y).$$
(2) For all $x \in T$ and all $n_1,n_2 \in \N_0^k$ with $d(x)
= n_1+n_2$ there exist unique $x_1,x_2 \in T$ with $x = x_1 x_2$ satisfying $d(x_1)=n_1$ and $d(x_2)=n_2$.

The map $d$ is called the {\em degree map}.}
\end{definition}

Often we shall call a $k$-semigraph $T$ just a semigraph when $k$ is unimportant or clear from the context.
We shall occasionally denote the degree $d(t)$ of an element $t$ in a $k$-semigraph also by $|t|$.
We denote the set of all elements of $T$ with degree $n$
by $T^{(n)}$ ($n \in \N_0^k$).
The cut-down $k$-semigraph $T^{(\le n)}$
is the $k$-semigraph consisting of all elements of $T$ with degree less or equal to
$n$.



\begin{definition}
{\rm
If $x \in T$ and $0 \le n_1
\le n_2 \le d(x)$ then there are unique $x_1,x_2,x_3 \in T$ such that $x = x_1
x_2 x_3$, $d(x_1)=n_1, d(x_2)= n_2-n_1$ and $d(x_3) = d(x)-n_2$.
$x_2$ will
be denoted by $x(n_1,n_2)$.
}
\end{definition}


\begin{definition} 
{\rm A $k$-semigraph $T$ is called {\em finitely aligned} if for all
$x,y \in T$ the
{\em minimal common extension} of $x$
and $y$, which is the set
\begin{eqnarray*}
T^{(\min)}(x,y) &=& \{(\alpha,\beta) \in T \times T |\,
\mbox{$x \alpha$ and $y \beta$ are defined},\\
&& x \alpha = y \beta,\, d(x \alpha) =
d(x) \vee d(y) \},
\end{eqnarray*}
is finite. }
\end{definition}

The last definition is a straight generalisation of finitely alignment in graphs (\cite{raeburnsimsyeendFinitelyAligned}).


\begin{lemma} \label{faclosure}
Let $\Lambda$ be a finitely aligned semigraph. For every finite
subsets $E$ of $\Lambda$ there exists a finite subset $F$ of
$\Lambda$ containing $E$ such that the following implication holds.
$$
\Big ( x_1,x_2,y_1,y_2 \in F, \,\, d(x_1)=d(x_2),\,\,
d(y_1)=d(y_2),
\,\, (\alpha,\beta) \in \Lambda^{(\min)}(x_1,y_1) \Big)$$
\begin{equation}
\Longrightarrow \Big ( x_2 \alpha \in F \,\, \mbox{if}\,\, x_2
\alpha \mbox{ is defined}, \,\, y_2 \beta \in F \,\, \mbox{if} \,\, y_2
\beta \mbox{ is defined} \Big ) \label{faclosurecondition}
\end{equation}

\end{lemma}

\begin{proof}
If $\Lambda$ is a graph then this lemma is a restatement of \cite[Lemma 3.2]{raeburnsimsyeendFinitelyAligned}.
If $\Lambda$ is a graph $\Gamma$ without the idempotent set $\Gamma^{(0)}$, so $\Lambda=\Gamma\backslash \Gamma^{(0)}$, then the assertion of this lemma follows also directly from \cite[Lemma 3.2]{raeburnsimsyeendFinitelyAligned} by applying it to the graph $\Gamma$.
If $\Lambda$ is none of these cases then this lemma
may be proved along the lines of \cite[Lemma 3.2]{raeburnsimsyeendFinitelyAligned}
with obvious adaption: one always takes into account whether a given product
in $\Lambda$ is defined and restricts to the defined products.
For example, instead of the definition of the set $E_i$ given in
the proof of \cite[Lemma 3.2]{raeburnsimsyeendFinitelyAligned}, one uses
\begin{eqnarray*}
E_i &=& \{\,x=\lambda_1(0,d(\lambda_1)) \ldots \lambda_j ( d(\lambda_{j-1}),d(\lambda_j)) \, |\,
\lambda_l \in \vee E_i,\\
&& x \mbox{ exists}, \, d(\lambda_l ) \le d(\lambda_{l+1}) \mbox{ for } 1 \le l \le j \, \}
\end{eqnarray*}
\end{proof}

%

\begin{definition}
{\rm
$T$ is called a {\em non-unital} $k$-semigraph if there exists a $k$-semigraph $T_1$ which has a unit $1 \in T_1$ such that
$T= T_1 \backslash \{1\}$.
}
\end{definition}

Suppose that $T$ is a non-unital $k$-semigraph.
Then $d(1)=0$ in $T_1$ since we have $d(1)=d(11)= d(1)+d(1)$.
Moreover, by the unique factorisation property in $T_1$
the identity $t = 1 t = t 1$ yields that $1$ is the only element in $T_1$ which has degree zero.
Consequently we have
$d(t)>0$ for all $t \in T$.

\section{The degree of a word}  \label{SectionDegreeWord}


The setting of this section is as follows.
$\calp$ is a set and $\calt$ is a $k$-semigraph or a non-unital $k$-semigraph.
%
%
%
%
%
$\F$ denotes the free non-unital $*$-algebra generated by the letter set $\calt \cup \calp$.
In other words, $\F$ is the vector space over the complex numbers
with base being all non-empty formal {\em words}
$a_1^{\epsilon_1} \ldots a_n^{\epsilon_n}$
($n \ge 1$) in the letters $a_i \in \calt \cup \calp$. Here $\epsilon_i \in \{1,*\}$.
Multiplication and taking adjoints within $\F$ is done formally.

\begin{definition} \label{degreeMap}
{\rm
The {\em degree}
$d(x)$ of a {word} $x=x_1 \ldots x_n$ in $\F$ ($n \ge 1$, $x_i
\in \calp \cup \calt \cup \calp^* \cup \calt^*$) is defined to be
$$d(x)=d(x_1) + \ldots + d(x_n),$$
where $d(x_i)$ is to be the
semigraph-degree $d(x_i)$ when $x_i \in \calt$, $d(x_i) = 0$ if $x_i \in
\calp$, and $d(x_i^*) = - d(x_i)$ for any $x_i \in \calt \cup \calp$.
}
\end{definition}

Since this degree map extends the degree map for $\calt$, we use the same notation $d$.
Note that the last definition is unambiguous: by the unique factorisation property in $\calt$ we have
$$d(x) = d(st) = d(s) + d(t)$$
for any decomposition $x =s t$ of $x,s,t \in \calt$ in $\calt$, and this is all we had to check.
The degree map satisfies the following formulas:
$$d(w v) = d(w)+ d(v) \qquad \mbox{and} \qquad d(w^*) = -d(w)$$
for all nonzero words $w$ and $v$ with $wv \neq 0$ in the first identity.
In general we may call such a map a degree map, even without the special form given in Definition
\ref{degreeMap}.
Note also that in order that Definition \ref{degreeMap} is without contradiction we need to have
that the intersection $\calp \cap \calt$, if non-empty, is a subset of $\calt^{(0)}$.

In this chapter we shall write $W_n$ for the words with degree $n \in \Z$.
%
%
Having the degree map $d$, we may write $\F$ as a direct sum of fibers,
where a fiber $\F_n$ is the linear span of all words with degree $n$, i.e.
we may write
\begin{equation} \label{directsumrep}
\F \cong \bigoplus_{n \in \Z^k} \F_n = \bigoplus_{n \in \Z^k} \lin(W_n).
\end{equation}

\begin{definition}  \label{defFiberSPace}
{\rm
We call the set
$F= \bigcup_{n \in \Z^k} \F_n = \bigcup_{n \in \Z^k} \lin(W_n)$
the {\em fiber space of $\F$}.
}
\end{definition}

Consider a quotient $X=\F/\I$ of $\F$. A word $w \in \F$ will also be called
a word in the quotient $X$, so $w + \I$ is called a word in $X$ when $w$ is a word in $\F$.
Assume that we are given a two-sided self-adjoint ideal $\I$ of $\F$ which is generated by
some subset of the fiber space. Then the quotient $\F/\I$ inherits the degree map from $\F$ as we are going to prove:

\begin{lemma}   \label{lemmadegreemap}
The degree map $d$ for words in $\F$ induces a well defined degree
map for the nonzero words in $X$
when $X$ is a quotient of $\F$ by a subset of the fiber space.
(Formula: $d(x+\I)= d(x)$.)
\end{lemma}

\begin{proof}
For two nonzero words $v+\I = w + \I$ in $X$, where $v$ and $w$ are words in $\F$,
we need to show that $d(v+\I):= d(v) = d(w) =: d(w + \I)$.
We have $v-w \in \I$.
Thus there are scalars
$\alpha_i \in \C$, words $a_i, b_i \in \F$, and elements $x_i \in \F_{j_i} \cap
\I$ such that
\begin{equation}   \label{vwfibers}
v-w = \sum_{i=1}^n \alpha_i a_i x_i b_i,
\end{equation}
where each summand $\alpha_i a_i x_i b_i$ is obviously in the fiber $\F_{j_i+d(a)+d(b)}$.
%
%
Since $v$ and $w$ are words, and thus elements of single fibers, say $v \in \F_{m_1}$ and $y \in \F_{m_2}$,
a compare of fibers in (\ref{vwfibers}) and using the direct sum representation (\ref{directsumrep}) shows that either $d(w) =d(v)$ (what we wanted to prove) or both $v$ and $w$ are elments in $\I$, which means that $v+\I$ is zero in $X$
(the case we exclude).
\end{proof}


\begin{definition}   \label{defFiberSpaceX}
{\rm
The {\em fiber space of $X$} is the image of the fiber space of $\F$ under the quotient map
$\F \longrightarrow X$.
}
\end{definition}

\begin{definition}   \label{definitionGaugaActions}
{\rm
Let $\sigma: \T^k \longrightarrow \aut(\F)$ be the {\em gauge action}
defined by
\begin{equation}   \label{gaugeactionformulas}
\sigma_\lambda (p) = p \qquad \mbox{and} \qquad
\sigma_\lambda(t) = \lambda^{d(t)} t
\end{equation}
for all $p \in \calp,t \in \calt$ and $\lambda \in \T^k$.
}
\end{definition}

The gauge action carries over to a canonically in the same way defined gauge action
$\sigma': \T^k \longrightarrow \aut (X)$. Indeed, since $X = \F/\langle Y \rangle$ is the quotient of $\F$ by a subset $Y$ of the fiber space,
and each element $r \in \F_m$ of the fiber space satisfies $\sigma_\lambda(r) = \lambda^{m} r$,
one has $\sigma_\lambda(\langle Y \rangle) \subseteq \langle Y \rangle$.
Hence $\sigma_\lambda$ induces $\sigma_\lambda' : \F/ \langle Y \rangle \longrightarrow \F/\langle Y \rangle$.
Note also that $\sigma_\lambda^{-1} = \sigma_{\lambda^{-1}}$ and similarly so for $\sigma_\lambda'$.
For simplicity we shall denote the gauge action on $X$ also by $\sigma$ if there is no danger of confusion.

\begin{definition}  \label{defGaugeActionOnX}
{\rm
Let
$\sigma: \T^k \longrightarrow \aut (X)$ denote the gauge action on $X$
determined by the formulas (\ref{gaugeactionformulas}).
}
\end{definition}

\begin{lemma}   \label{lemmaGaugeActions}
$X$ is the $*$-algebraic quotient of $\F$ by a subset of the fiber space if and only if
there is a gauge actions on $X$ as defined in Definition
\ref{defGaugeActionOnX}.
\end{lemma}

\begin{proof}
One direction we have proved.
For the reverse direction assume that $X$ has a gauge action.
Write $X = \F/\I$ canonically for a two-sided self-adjoint ideal
$\I$ in $\F$.
Let $x$ be an arbitrary element of $\I$. We may write $x= \sum x_n$
for $x_n$ in the fiber $\F_n$ for all $n \in \Z^k$.
We have, in $X$,
$$0 = \sigma_\lambda (x +\I)=\sum_{n \in \Z^k} \lambda^{n} x_n + \I$$
for all $\lambda \in \T^k$. Thus $x_n \in \I$ for all $n
\in\Z^k$, and so $x_n \in \I \cap \F_n$.
Since $x$ was arbitrary, $\I$ is the linear span of $\bigcup_{n \in \Z^k} (\I \cap \F_n)$.
So $X$ is the quotient of a subset of the fiber space.
\end{proof}

Since it is somewhat shorter, we shall occasionally write $|x|$ for the
degree $d(x)$.

\section{Semigraph algebras}

We shall use the following notions when we speak about algebras.
A {\em $*$-algebra} means an algebra
over $\C$ endowed with an involution. An element $s$ in a
$*$-algebra is called a {\em partial isometry} if $s s^* s= s$,
and a {\em projection} $p$ is an element with $p = p^2 = p^*$.
If $I$ is a subset of a $*$-algebra then $\langle I \rangle$ denotes the self-adjoint two-sided
ideal generated by $I$ in this $*$-algebra.

\begin{definition}[Semigraph algebra]    \label{DefSemigraphAlgebra}
{\rm A {\em $k$-semigraph algebra} $X$ is a $*$-algebra which is
generated by disjoint subsets $\calp$ and $\calt$ of $X$, where

\begin{itemize}
\item[(i)]
$\calp$ is a set of commuting projections closed under taking
multiplications,

\item[(ii)]
$\calt$ is a set of nonzero partial isometries closed under nonzero
products,

\item[(iii)]
%
$\calt$ is a non-unital finitely aligned $k$-semigraph,

\item[(iv)]
for all $x \in \calt$ and all $p \in \calp$ there is a $q \in \calp$
such that $p x = x q$,

\item[(v)]
for all $x,y \in \calt$ there exist $q_{x,y,\alpha,\beta} \in \calp$
such that
\begin{equation}  \label{defExpansion}
x^* y = \sum_{(\alpha,\beta) \in \calt_1^{(\min)}(x,y)} \alpha q_{x,y,\alpha,\beta} \beta^*, \mbox{ and}
\end{equation}

\item[(vi)]
$X$ is canoncially isomorphic to the quotient of $\F$ by a subset of the fiber space
(Definition \ref{defFiberSPace}).
\end{itemize}
}
\end{definition}

We denote the unitization of the non-unital semigraph $\calt$ by $\calt_1$ ($\calt_1$ appears in identity (\ref{defExpansion})).
Note that
$\calt_1 := \calt \sqcup \{1\}$ is a finitely aligned $k$-semigraph by Definition \ref{DefSemigraphAlgebra} (iii).
It is understood in (\ref{defExpansion}) that the unit $1$ of $\calt_1$
is also a unit for $X$.
So we may assume that $1$ is the unit
of the unitization of $X$, which is $\tilde X = X \oplus \C 1$.
The only reason why we use non-unital $k$-semigraphs instead of $k$-semigraphs is that we wanted to avoid forcing a semigraph algebra to be unital.


Note that the product of two elements $s$ and $t$ of $\calt$ stays in $\calt$ (so is composable in $\calt$)
if and only if $s t \neq 0$. This is a somewhat subtle implication of Definition \ref{DefSemigraphAlgebra} (ii).
In general, such a construction is a typical example of a semimultiplicative set; for instance if $R$ is a ring, then $R\backslash \{0\}$
is a semimultiplicative set under multiplication.

We shall occasionally write line (\ref{defExpansion}) as
\begin{equation} \label{formula2}
x^* y = \sum_{x  \alpha = y \beta} \alpha q_{x,y,\alpha,\beta} \beta^*.
\end{equation}
Note that in the last identity $d(x \alpha)= d(x) + d(\alpha) = d(y \beta) = d(y) + d(\beta)$,
so that in formula (\ref{formula2}) we have
$$d(x^* y) = -d(x) + d(y) = d(\alpha) -d(\beta) = d(\alpha q_{x,y,\alpha,\beta} \beta^*).$$
This shows that formula (\ref{formula2}) is a relation in the fiber space.

The precise meaning of point (vi) of Definition \ref{DefSemigraphAlgebra}
is that the kernel of the canoncial epimorphism $\F \longrightarrow X$ is an ideal which is generated by a certain subset of the fiber space.
In other words, $X$ can be regarded as the free $*$-algebra $\F$ generated by $\calt$ and $\calp$ divided by a family of equations
$x_i = 0$, where $x_i$ is a linear combination of words $w$ with common degree $d(w)$ (depending only on $i$).
Equivalently, there is a gauge action on $X$ (Lemma \ref{lemmaGaugeActions}).

\begin{lemma}
There is a degree map on the set of nonzero words of $X$ which extends the degree map on $\calt$
(see Definition \ref{degreeMap}).
\end{lemma}

\begin{proof}
This is Lemma \ref{lemmadegreemap} in combination with Definition \ref{DefSemigraphAlgebra} (vi).
\end{proof}






\begin{lemma}    \label{lemmaSemigrafElementary}
Let $X$ be a semigraph algebra.
Then

(i) $d(1) = 0$,
(ii) $d(t) > 0$ for all $t \in \calt$,
(iii) $t^* t \in \calp$ for all $t \in \calt$, and
(iv) $s^* t = \delta_{s}^t s^* s$ for all $s,t \in \calt$ with $d(s) = d(t)$.
%
%
%
\end{lemma}

\begin{proof}
(i) and (ii) were proved in Section \ref{SectionSemigraph}.
%

(iii)-(iv) If $d(s) = d(t)$ then $s^* s = 1 q_{s,t,1,1} 1 \in \calp$ and $s^* t = 0$ for $s \neq t$ by (\ref{defExpansion}).
\end{proof}

\begin{definition}
{\rm
The enveloping $C^*$-algebra $C^*(X)$ of $X$ is called the {\em semigraph $C^*$-algebra associated to $X$}.
}
\end{definition}

\begin{lemma}
Points (ii), (iv) and (v) of Definition \ref{DefSemigraphAlgebra}
also hold for $x,y \in \calt_1$.
\end{lemma}

\begin{proof}
(iv) Of course, $p 1 = 1 p$. (v) Say $x=1$.
Then $(y,1)$ is the only element in $\calt_1^{\rm (min)}(1,y)$
and one has formula (\ref{defExpansion}), namely $1^* y =  y (y^* y) 1$ with $q_{1,y,y,1}=y^* y \in \calp$
(Lemma \ref{lemmaSemigrafElementary} (iii)).
\end{proof}

\begin{definition}
{\rm
We shall use the following notations:
\begin{eqnarray*}
\calt_1 \calp &=& \set{ s p \in X}{s \in \calt_1, p \in\calp}, \\
\calt_1 \calp \calt_1^*  &=& \set{ s p t^* \in X}{ s,t \in \calt_1, p \in\calp}.
\end{eqnarray*}
}
\end{definition}

\begin{definition}
{\rm
We shall use the following vocabulary for better readability of this paper:

We call an element of $\calt_1 \calp \calt_1^*$ a {\em standard word}
(of the semigraph algebra $X$).
We call an element of $\calt_1 \calp$ a {\em half-standard word}.
}
\end{definition}

So an element $w$ of a semigraph algebra is a standard word if it allows a representation $w = s p t^*$
for some $s,t \in \calt_1$ and $p \in \calp$. In particular, $p, s p$ and $p t^*$ are also standard words (since $1 \in \calt_1$).
A half-standard word is a standard word.

The first important observation we shall make is that
the word set of a semigraph algebra is an inverse semigroup.
%
Note in particular that this also means that the range and source projections of all words commute among each other (also between different words).
It is however not true that the standard words form an inverse semigroup. They do not form a stable set under multiplication.


\begin{lemma} \label{semigraphwordrep}
\begin{itemize}
\item[(a)]
The word set of $X$ is an inverse semigroup of partial isometries.

\item[(b)]
For each word $w$ there are half-standard words $a_i,b_i$ and $c_i$ such
that
\begin{equation}   \label{sumrepcommute}
w w^* = \sum_{i=1}^n a_i a_i^* \qquad \mbox{and} \qquad w = \sum_{j =1}^m b_j
c_j^*
\end{equation}
with $d(w)=d(b_j c_j^*)$ for all $1 \le j \le m$.
\end{itemize}
\end{lemma}

\begin{proof}
We are going to show that range projections of half-standard words commute.
Let $a$ and $b$ be half-standard words.
Then we may choose $x,y \in \calt_1$ and $p,q \in \calp$ such that $a = x p$ and $b = y q$.
We have
$$a a^* b b^* = x p (x^* y) q y^* = \sum_{(\alpha,\beta) \in \calt_1^{(\min)}(x,y) } x p \alpha
q_{x,y,\alpha,\beta} \beta^* q y^*$$
$$ = \sum x \alpha
p_\alpha q_{x,y,\alpha,\beta} q_\beta \beta^* y^*
=
\sum y \beta
p_\alpha q_{x,y,\alpha,\beta} q_\beta \beta^* y^*
$$
for certain
$p_\alpha ,q_\beta \in \calp$ such that $p \alpha = \alpha p_\alpha,
q \beta = \beta q_\beta$ by Definitions \ref{DefSemigraphAlgebra} (iv) and (v),
and since $x \alpha = y \beta$.
We see by the above identity that $a a^* b b^*$ is self-adjoint (since $\calp$ is a commuting set,
Definition \ref{DefSemigraphAlgebra} (i)). Thus
$a a^*$ and $b b^*$ commute.


%
%


We are going to show the first identity in (\ref{sumrepcommute}). We shall prove it by induction on the length of the word $w$. Assume that $w w^* = \sum_i a_i a_i^* = \sum_i x_i q_i x_i^*$ is proved (for $x_i \in \calt_1$ and $q_i \in \calp$ with $a_i = x_i q_i$).
If $t \in \calt$ then $t w w^* t^* = \sum_i t x_i q_i x_i^* t^*$
and we are done with this inductive step. If $p$ is in $\calp$ then
$$p w w^* p^* = \sum_i p x_i q_i x_i^* p^*= \sum_i  x_i p_i' q p_i' x_i^*$$
for the $p_i' \in \calp$ of Definition \ref{DefSemigraphAlgebra} (iv) satisfying $p x_i = x_i p_i'$,
and so we are also done with this inductive step.
If $t \in \calt$ then
$$t^* w w^* t = \sum_i t^* x_i q_i x_i^* t = \sum_i (t^* x_i) q_i (t^* x_i)^*$$
$$= \sum_i \sum_{t \alpha_i = x_i \beta_i} \alpha_i q_{t,x_i,\alpha_i,\beta_i} \beta_i^* q_i \sum_{t \alpha_i' = x_i \beta_i'}
\beta_i' q_{t,x_i,\alpha_i',\beta_i'} {\alpha_i'}^*$$
$$=
\sum_i \sum_{t \alpha_i = x_i \beta_i} \alpha q_{t,x_i,\alpha,\beta} \beta^*
\beta q_{i,\beta} q_{t,x_i,\alpha,\beta} {\alpha}^* $$
by Definition \ref{DefSemigraphAlgebra} (v) in the second line, and Definition \ref{DefSemigraphAlgebra} (iv) ($q_i \beta_i' = \beta' q_{i,\beta_i'}$) and Lemma \ref{lemmaSemigrafElementary} (iv) in the third line.
Note here also that $\alpha' = \alpha$ since necessarily $\beta' = \beta$, and one $\beta'$ allows only one solution $\alpha'$ in the equation $t \alpha' = x_i \beta_i$
by the unique factorisation property (Definition \ref{Defsemigraph}).
This proves the inductive step also in this case.

The proof of the second sum in (\ref{sumrepcommute}) is very similar.

By the first formula of (\ref{sumrepcommute}) and the fact that the $a_i a_i^*$'s commute (as we have proved at the beginning of this lemma) it is evident that $w w^*$ and $v v^*$ commute
for all words $w$ and $v$. Now $\calp$ and $\calt$ consist of partial isometries. Hence also their compositions are partial isometries since their source and range
projections commute. And so further we see by induction that words of any length are partial isometries. This proves point (a).
\end{proof}

\begin{corollary}   \label{corollarySpannedStandardWords}
A semigraph algebra is spanned by its standard words. 

The range projection of a word is a sum of range projections of half-standard words.
\end{corollary}

\begin{proof}
This is a restatement of Lemma \ref{semigraphwordrep} (b).
\end{proof}

\begin{corollary}   \label{corollaryInvSemigroup}
A semigraph algebra is generated by the inverse semigroup of all its words.
\end{corollary}

\begin{proof}
The semigraph algebra is the linear span of its words, and the word set is an inverse semigroup
by Lemma \ref{semigraphwordrep}.
\end{proof}

\begin{lemma} \label{lemmaproductprojection}
If $v_1, \ldots, v_n$ are half-standard words then there are half-standard words
$w_1 \ldots,
w_m$ such that
\begin{equation}  \label{Pproduct}
P_{v_1} P_{v_2} \ldots P_{v_n} = P_{w_1} + \ldots + P_{w_m},
\end{equation}
where the $P_{w_k}$'s are mutually
orthogonal and the $w_k$'s have common degree, i.e. $d(w_k)= d(w_i)$
for all $1 \le k ,i \le m$.
\end{lemma}

\begin{proof}
By induction hypothesis
assume that (\ref{Pproduct}) is already proved.
We may write $v_{n+1} = s p$ and $w_k = t_k q_k$ for some $s, t_k \in \calt_1$ and $p,p_k \in \calp$.
Assume that $d(t_k) = d(t_i)$ for all $k$ and $i$
and that the $t_k$'s are mutually distinct.
By Definition \ref{DefSemigraphAlgebra} (iv) and (v) we have
%
\begin{eqnarray*}
P_{v_{n+1}} \sum_{k=1}^m P_{w_k}
&=& \sum_{k=1}^m (s p s^*) (t_k q_k t_k^*)\\
&=& \sum_{k=1}^n s p \sum_{s \alpha_k = t_k \beta_k} \alpha_k q_{k,s,t_k,\alpha_k,\beta_k} \beta_k^* q_k t_k^*\\
&=&  \sum_{k=1}^n \sum_{s \alpha_k = t_k \beta_k} s \alpha_k p_{k,\alpha_k} q_{k,s,t_k,\alpha_k,\beta_k} q'_{k,\beta_k} \beta_k^* t_k^*\\
&=&  \sum_{k=1}^n \sum_{s \alpha_k = t_k \beta_k} P_{t_k \beta_k p_{k,\alpha_k} q_{k,s,t_k,\alpha_k,\beta_k} q'_{k,\beta_k}},
\end{eqnarray*}
where $p_{k,\alpha_k} ,q'_{k,\beta_k}  \in \calp$ such that $p \alpha_k = \alpha_k p_{k,\alpha_k}$ and $q_k \beta_k = \beta_k q'_{k,\beta_k}$.
This proves the claim since the $t_k \beta_k$'s are mutually distinct.
\end{proof}

\begin{lemma}  \label{lemmaSourceProjHSword}
The source projection of a half-standard word is in $\calp$.
\end{lemma}

\begin{proof}
Let $\alpha p$ ($\alpha \in \calt_1, p \in \calp$) be a half-standard word.
Then $p \alpha^* \alpha p \in \calp$ by Lemma \ref{lemmaSemigrafElementary}
(iv) and Definition \ref{DefSemigraphAlgebra} (i).
\end{proof}

The idempotent elements of the inverse semigroup of words in a semigraph algebra are
the range projections of words.
By Lemma \ref{semigraphwordrep} the range projection, and thus also the source projection, of any word is the orthogonal sum of range projections of half-standard words.
It is thus natural to consider common refinements of such range projections in the further analysis,
and this is what the next definitions are all about.
These common refinements will be called standard projections.
They will be useful in the further analysis of semigraph algebras.

\begin{definition}
{\rm
For an element $x$ of a $*$-algebra $X$ we put
$P_x = x x^*$ and $Q_x = x^* x$.
For a subset $Z$ of a $*$-algebra $X$ we set
\begin{eqnarray*}
\P(Z) &=& \set{P_{x} (1-P_{y_1}) \ldots (1-P_{y_m}) \in X}{ x,y_i \in Z,\,  m \ge 0}.
\end{eqnarray*}
}
\end{definition}

\begin{definition}
{\rm
For better readability we introduce the following vocabulary.

We call an element of $\P(\calt_1 \calp)$ a {\em standard projection}
(of the semigraph algebra $X$).

We call an element of $\P(\calp)$ an {\em elementary standard projection}.
}
\end{definition}

For instance, $x p x^*(1 - y q y^*)$ is a standard projection given $x,y \in \calt_1$ and $p,q \in \calp$.

\section{The core}
\label{SectionsSemigraphAlgebraCore}


\begin{definition}
{\rm
The {\em core} of a semigraph algebra $X$ is the $0$-fiber $X_0$, that is, the linear span of all words with degree zero.
}
\end{definition}

Since we have $d(v w) = d(v) + d(w)$ and $d(v^*) = -d(v)$ for words $v$ and $w$ in $X$, the core is even
a $*$-subalgebra of $X$.
The next proposition is the basic tool for the analysis of the core.

\begin{proposition} \label{centralAFproposition}
(i) Suppose $X$ is a $*$-algebra and $G=\{s_1,\ldots,s_n\}$ a finite, self-adjoint
subset of partial isometries of $X$ with commuting range projections.
Let $\{p_1,\ldots,p_N \}$ be the collection of all minimal projections
of the finite dimensional commutative subalgebra $Z$ of $X$ generated
by the range projections $\{P_{s_i}\}_{i=1}^n$ of the elements of $G$.

Assume that for all $1 \le i,j \le n$ there exist nonnegative reals
$\lambda_1,\ldots,\lambda_n \ge 0$ such that
\begin{equation}    \label{finiteM}
s_i s_j =
\sum_{k=1}^{n} \lambda_k s_{k}.
\end{equation}
Assume that for all $1 \le i,j \le
n$
\begin{equation}   \label{alsoInZ}
s_i P_{s_j} s_i^* \in Z.
\end{equation}

Then for all $1 \le x,y \le N$ and all $1 \le i \le n$ one has
\begin{equation}  \label{matrixUnits}
p_x s_i p_y
\neq 0 \qquad \Longrightarrow \quad p_x s_i p_y = p_x s_i = s_i p_y.
\end{equation}

(ii) Assume further that for all $1 \le x,y \le N$ and all $1 \le
i,j \le n$ one has
\begin{equation} \label{assumption1}
p_x s_i p_y \neq 0  \quad \mbox{and} \quad p_x s_j p_y\neq 0
\qquad \Longrightarrow \qquad  p_x s_i s_j p_x  \mbox{ is a projection}.
\end{equation}
Then the linear span $M$ of $G$ is a finite dimensional $C^*$-algebra with
generating canonical matrix units $(e_{x,y})_{1 \le x,y \le N}$,
where $e_{x,y} = p_x s_i p_y$ when there is some $1 \le i \le n$
such that $p_x s_i p_y \neq 0$, and $e_{x,y}= 0$ otherwise. Actually, $e_{x
y}$ does not depend on $i$. If $e_{x,x} \neq 0$ then $e_{x,x} =
p_x$.
\end{proposition}

Note that by formula (\ref{finiteM}) and the fact that
$G$ is self-adjoint and finite, $M$ of the last proposition is surely a finite dimensional $*$-algebra.
The point is that $M$ is even a $C^*$-algebra together
with the relatively explicit computation of the matrix units.
Note also that the minimal projections $p_x$ are just the common refinements of the projections $P_{s_i}$.

\begin{proof}
Since $G$ is self-adjoint (that is, $G^*
\subseteq G$) $Z$ contains also the source projection of every
element of $G$. We have $s_i Z s_i^* \subseteq Z$ by (\ref{alsoInZ}) for every $s_i \in G$.
This also implies $s_i^* Z s_i \subseteq Z$ as $s_i^* \in G$
too. In particular, $s_i p_x s_i^* \in Z$ for all $1 \le x \le N$ and $1 \le i \le n$.
Thus we have
$$p_x s_i p_y p_y^* s_i^* p_x^* \in p_x s_i Z s_i^* p_x \subseteq p_x Z p_x \subseteq Z$$
for all $1 \le x \le N$ and all $1 \le i \le n$.
Since $p_x$ and $s_i$ are partial isometries with commuting source projection $p_x$ and range projection
$s_i s_i^*$, $p_x s_i$ is also a partial isometrie. By such considerations we check
that
$$s_i, p_x s_i, s_i s_j, p_x s_i p_y$$
are partial
isometries with source and range projections living in the
commutative algebra $Z$.
The partial isometry
$$e_{xy}^i := p_x s_i p_y$$
is
either zero or has source projection $p_y$ and range projection
$p_x$ by minimality of the $p_z$'s. Hence, the composition $e_{x
y}^i e_{x' y'}^{i'}$ is a partial isometry again for all $1 \le x,y,x',y' \le N$ and all $1 \le i, i' \le n$.
If $e_{x y}^i \neq 0$ then, since $Z$ is a commutative algebra,
\begin{equation}   \label{represei}
e_{x y}^i = p_x s_i p_y p_y s_i^* s_i
p_y = p_x (s_i p_y s_i^*) s_i = p_x s_i,
\end{equation}
as the minimal projection $p_x$ absorbs the projection $s_i p_y
s_i^* \in Z$.
This proves claim (\ref{matrixUnits}).

Hence, if $e_{x y}^i \neq
0$ and $e_{x y}^j \neq 0$, then
$$e_{x y}^i (e_{x y}^j)^* = p_x s_i
s_j^* p_x.$$
By assumption (\ref{assumption1}) (also recall that $s_j^* \in G$),
this is a projection. Since, as noted above, $e_{x y}^i$ and $e_{x y}^j$
have common source projection $p_y$ and range projection $p_x$, this
is only possible when $e_{x y }^i = e_{x y}^j$. This proves that
$e_{x y} := e_{x y }^i$,
if nonzero, does not depend on $i$.

$M$ is a finite dimensional $*$-algebra by assumption (\ref{finiteM}). Since $\sum_{i=1}^N
p_i$ is a unit of $M$, the collection of all $e_{xy}$'s span $M$.
We have to show that the linear map $\varphi:M \rightarrow M_N(\C)$
determined by $\varphi(e_{ij})={ \hat e}_{ij}$ for $e_{ij} \neq 0$,
where ${\hat e}_{ij}$ denote the canonical matrix units of
$M_N(\C)$, is a $*$-homomorphism.  It will then automatically follow that $\varphi$ is
injective. That the nonzero $e_{x y}$'s are linearly independent,
follows from a standard proof exploiting the above mentioned fact
that the source and range projections of $e_{x y}$ are $p_y$ and
$p_x$, respectively. Suppose that $e_{xy}$ and $e_{yz}$ are nonzero.
Then
$$0 \neq e_{xy}^*=p_y s_i^* p_x = e_{y x}$$
by the above proved
uniqueness of $e_{y x}$. Thus $\varphi(e_{xy}^*)=\varphi(e_{xy})^*$.
Now we have
$$e_{xy} e_{yz} = p_x s_i s_j p_z = \sum_{k=1}^n \lambda_k
p_x s_k p_z= \lambda e_{xz}$$
for certain $\lambda_k \ge 0$ and $\lambda \ge 0$
by (\ref{represei}), by assumption (\ref{finiteM}), and by the above proved uniqueness of
the $e_{xz}^k$'s. Since both $e_{xy} e_{y z}$ and $e_{xz}$ are nonzero
partial isometries (as mentioned above), $|\lambda|^2 =1$, and so
$\lambda=1$ as $\lambda>0$.
Hence we have $e_{xy} e_{yz} = e_{xz}$, and so
$$\varphi(e_{xy}
e_{yz}) = \varphi(e_{xz}) = \hat e_{xz} =  \hat e_{xy} \hat e_{yz} =  \varphi(e_{xy}) \varphi( e_{yz}).$$
\end{proof}

For the remainder of this section assume that we are given a semigraph algebra $X$.

\begin{lemma} \label{finiteprodsetXe}
For every finite set $D$ of half-standard words there exists a finite
set $H$ of half-standard words containing $D$ such that
\begin{equation} \label{exprG}
G=\set{x y^*
\in X }{x,y \in H,\, d(x)=d(y)}
\end{equation}
satisfies all assumptions stated in
Proposition \ref{centralAFproposition}.
\end{lemma}

\begin{proof}
Let $D$ be a finite set of half-standard words. There are
$x_1,\ldots,x_n \in \calt_1$ and $l_1, \ldots, l_n \in \calp$ such
that $D=\{x_1 l_1,\ldots, x_n l_n\}$. By Lemma \ref{faclosure}
there exists a finite subset $F$ which contains $E=\{x_1,\ldots, x_n\}$
and satisfies the stability condition (\ref{faclosurecondition}). By Definition \ref{DefSemigraphAlgebra} (v),
for every $x,y \in F$  
we may choose $q_{x,y,\alpha,\beta} \in \calp$ such that
$$x^* y =
\sum_{(\alpha,\beta) \in \calt_1^{(\min)}(x,y)} \alpha
q_{x,y,\alpha,\beta} \beta^*.$$
Write $L$ for the finite collection
of these $q_{x,y,\alpha,\beta}$'s. Set
$\cala = \bigcup_{i \in k} \calt^{(e_i)}$.
Define the
following finite letter set $A$,
$$A= \set{a \in \cala}{ \exists \alpha,\beta \in \calt_1 \mbox{ such that } \alpha a
\beta \in F}.$$
In other words, $\cala$ is the collection of those letters which are part of a word in $F$.
Consequently, for every $x \in F$ there are $a_j \in A$
such that $x = a_1 \ldots a_i$. For every $q \in \calp$ and every $a
\in A$ choose a projection $Q(a,q)$ in $\calp$ such that $q a = a Q(a,q)$
according to Definition \ref{DefSemigraphAlgebra} (iv). Successively applying the last identity
we get
\begin{eqnarray}
q a_1 \ldots a_i &=& a_1 Q(a_1,q) a_2 \ldots a_i  \nonumber \\
&=&
a_1 \ldots a_i  Q(a_i,Q(a_{i-1}, \ldots Q(a_2, Q(a_1,q))\ldots )).  \label{seq22}
\end{eqnarray}
Define
$L_0$ to be the finite set
$$L_0 = \{1 ,l_1,\ldots,l_n\} \cup L.$$
For $n
\in \N^k_0$ set
\begin{eqnarray*}
L_{n} &=& L_0 \cup \{\, Q(a_i,Q(a_{i-1}, \ldots Q(a_2,
Q(a_1,q))\ldots
))  \in \calp \, | \\
&& i \in \N, \,a_1,\ldots,a_i \in A,\, |a_1|+ \ldots +|a_i| \le n,\,
q \in L_0 \,\}.
\end{eqnarray*}
Define $\Pi_n$ to be the set of all finite products of elements of $L_n$.
That is, an element of $\Pi_n$ is a finite product of projections which arise as projections from $L_0$ which then skip at most $n$ letters of $A$ (which is the letter set for the words in $F$).
Notice that since $L_0$ is a finite set, $L_n$ is a finite set. Thus, since $\calp$ is a commuting set, $\Pi_n$ is a finite set. 
Note also that $(R_n)_n$ and $(\Pi_n)_n$
are families of sets which increase in size.
Now define
$$H= \set{\alpha q \in \calt_1 \calp}{
\alpha \in F,\, q \in \Pi_{|\alpha|}} \backslash \{1\}.$$ Then $H$
is a finite set. It contains the set $D$ as $x_i \in E \subseteq F$ and $l_i \in L_0 \subseteq \Pi_{|x_i|}$ for every $1 \le i \le n$. This is the desired $H$ which appears in (\ref{exprG}).
Define $G$ as in (\ref{exprG}).

We aim to check that the requirements stated in Proposition
\ref{centralAFproposition} for a set $G$ there hold also for this $G$. Let us be given
$$g=
\alpha_1 q_1 q_2 \alpha_2^* \in G \qquad \mbox{and} \qquad h = \alpha_3 q_3 q_4
\alpha_4^* \in G,$$
where $\alpha_1 q_1 ,\ldots , \alpha_4 q_4 \in H$
with $\alpha_1,\ldots,\alpha_4 \in F$, $|\alpha_1|=| \alpha_2|$,
$|\alpha_3|=| \alpha_4|$, $q_1,q_2 \in \Pi_{|\alpha_1|}$ and
$q_3,q_4 \in \Pi_{|\alpha_3|}$. Then
\begin{eqnarray}
\nonumber
g h  &=& (\alpha_1 q_1 q_2 \alpha_2^* )(\alpha_3 q_3 q_4 \alpha_4^*)\\
\nonumber &=& \sum_{(x , y) \in \calt_1^{(\min)}(\alpha_2,\alpha_3)}
\alpha_1 q_1 q_2  x  q_{\alpha_2,\alpha_3,x,y}  y^*   q_3 q_4 \alpha_4^*\\
\label{prodghresult} &=& \sum_{(x , y) \in
\calt_1^{(\min)}(\alpha_2,\alpha_3)} \alpha_1 x q_1^{(x)}
q_3^{(y)}  y^* \alpha_4^*
\end{eqnarray}
where $q_1 q_2 x = x \tilde q_1$ for some $\tilde q_1 \in
\Pi_{|\alpha_1|+|x|}$ by (\ref{seq22}) (because $q_1 q_2$ consists of elements which at most skipped $|\alpha_1|$ letters of $A$, so $\tilde q_1$ consists of elements which at most skipped $|\alpha_1|+|x|$ letters of $A$), and where we have put $q_1^{(x)} = \tilde q_1
q_{\alpha_2,\alpha_3,x,y} \in \Pi_{|\alpha_1|+|x|}$.
Similarly we have $q_3
q_4 y = y q_3^{(y)}$ for some $q_3^{(y)} \in \Pi_{|\alpha_3|+|y|}$ by
(\ref{seq22}).

By condition (\ref{faclosurecondition}) of Lemma \ref{faclosure} we
have
\begin{equation}  \label{conclusionF}
\alpha_1 x \neq 0, \alpha_4 y \neq 0 \quad \Longrightarrow \quad
\alpha_1 x , \alpha_4 y \in F \quad \Longrightarrow \quad \alpha_1 x
q_1^{(x)}, \alpha_4 y q_3^{(y)} \in H.
\end{equation}
We have seen that $G$ is a finite, self-adjoint set such that for
all $g,h \in G$, $gh$ is the sum of certain elements in $G$ as we can see from expression (\ref{prodghresult}).
This fact proves the requirement (\ref{finiteM}) in Proposition \ref{centralAFproposition}.

By (\ref{prodghresult}) we have
\begin{eqnarray*}
(gh) (gh)^*  &=&
\sum_{( x' , y'),(x , y) \in
\calt_1^{(\min)}(\alpha_2,\alpha_3)}
\alpha_1 x q_1^{(x)} q_3^{(y)}  y^* \alpha_4^*
\alpha_4 y' q_3^{(y')} q_1^{(x')} x' \alpha_1\\
&=& \sum_{i=1}^m g_i g_i^*
\end{eqnarray*}
for some $g_i \in G$ as $y^* y' = \delta_{y}^{y'} Q_y$ in the above sum.
This exactly proves (\ref{alsoInZ}) of Proposition \ref{centralAFproposition}.


To prove (\ref{assumption1}) of Proposition
\ref{centralAFproposition}, we have to show that if $p,q$ are
minimal projections in $Z$ (the commutative algebra generated by the range projections of the elements of $G$), and $g,h \in G$ satisfy $p g q \neq 0$
and $p h q \neq 0$, then $p g h p$ is a projection.

We may assume $p gh p \neq 0$. We may write $g= \alpha_1 q_1 q_2
\alpha_2^* $ and $h =\alpha_3 q_3 q_4 \alpha_4^*$ as above. Then $g
h$ equals (\ref{prodghresult}). To analyse the sum
(\ref{prodghresult}), we consider $(x , y) \in
\calt_1^{(\min)}(\alpha_2,\alpha_3)$. Set $p_x = \alpha_1 x
q_1^{(x)} q_1^{(x)} x^* \alpha_1^*$. If $p_x\neq 0$ then $\alpha_1 x
\neq 0$ and so $p_x \in G$ by conclusion (\ref{conclusionF}). Thus $p_x = p_x p_x^*$ is an element of $Z$.
Note that the $p_x$'s
are mutually orthogonal for different $x$'s. Since $p$ is a minimal projection of $Z$, there is at most
one $x_0$ such that $p=p p_{x_0}  \neq 0$.
Comsequently, by (\ref{prodghresult}) we have
$$0 \neq p g h p = p p_{x_0} g h
p_{x_0} p = p p_{x_0} \alpha_1 x_0 q_1^{(x_0)} q_3^{(y_0)} {y_0}^*
\alpha_4^*  p_{x_0} p$$
$$ = p p_{x_0} \alpha_1 x_0 q_1^{(x_0)} q_3^{(x_0)} {x_0}^*
\alpha_1^*  p_{x_0} p = p p_{x_0} p_{x_0} p_{x_0} p = p$$
for $(x_0 , y_0) \in
\calt_1^{(\min)}(\alpha_2,\alpha_3)$, and 
where the facts ${y_0}^*
\alpha_4^*  p_{x_0} \neq 0$ and $|\alpha_1 x_0|=| \alpha_4 y_0|$ forces
the conclusion $\alpha_4 y_0 = \alpha_1 x_0$.
This shows that $p g h p$ is a projection.
We have proved that $G$ satisfies all the requirements stated in Proposition \ref{centralAFproposition},
and this was the claim.
\end{proof}

The next corollary is the main result of this section. The core is locally matrical
(i.e. the algebraic direct limit of finite dimensional $C^*$-algebras).

\begin{corollary} \label{corollarySemigraphMatrixunits}
The core is the union of a net of finite dimensional $C^*$-algebras,
each one allowing a matrix representation where each projection on
the diagonal is a finite sum of mutually orthogonal standard projections.
A $C^*$-representation of $X$ is injective on
the core if and only if it is non-vanishing on nonzero standard projections.
\end{corollary}

\begin{proof}
The core is the linear span of words with degree zero.
Thus, by Lemma \ref{semigraphwordrep}, the core is the linear span of words $x y^*$
where $x,y$ are half-standard words with degree $d(x)=d(y)$.
Let $f=\{x_1 y_1^*,\ldots,x_n y_n^*\}$ be a finite subset of the core with $d(x_i)=d(y_i)$.
Set $D=\{x_1,\ldots,x_n,y_1,\ldots,y_n\}$. Choose $G$ for $D$ according to Lemma \ref{finiteprodsetXe}.
Then $f \subseteq G$. The linear span of $G$ is a finite dimensional $C^*$-algebra
by Proposition \ref{centralAFproposition}. 
This finite dimensional $C^*$-algebra may be represented by a direct sum of matrices with diagonal 
entries $e_{x,x}=p_x$, where $p_x$ is a minimal projection of the commutative algebra generated
by the range projections $P_{g}$'s ($g \in G$).
Thus $p_x$ is a common refinement of such $P_g$'s, that means,
\begin{equation}  \label{p_xDarstellung}
p_x = P_{g_1}
\ldots P_{g_m} (1-P_{h_1}) \ldots (1-P_{h_l})
\end{equation}
for some $g_i,h_i \in G$.
Now an element of $G$ is of the form $x y^*$ ($x,y$ half-standard words), and so
$P_{x y^*} = x Q_y x^* = P_{z}$ for the half-standard word $z= x Q_y$.
Hence, if we expand $P_{g_1}
\ldots P_{g_m}$ in (\ref{p_xDarstellung}) according to Lemma \ref{lemmaproductprojection},
we see that $p_x$ is the orthogonal sum of standard projections.
This proves the first claim of the corollary.
The second claim is now clear, as a homomorphism defined on the core is injective
if it is non-vanishing on the nonzero standard  projections (thus non-vanishing on the matrix
diagonal entries).
%
\end{proof}

\section{The Cuntz--Krieger uniqueness theorem}
\label{SectionCuntzKriegerUniqueness}

In this section we are going to prove a Cuntz--Krieger uniqueness theorem
for a semigraph algebra. To this end we shall apply theorems
of our paper \cite{burgiCancelling}.

Let us recall what we need.
In \cite{burgiCancelling} we consider a $*$-algebra $X$
which is generated as a $*$-algebra by a subset $\cala$.
One has given an amenable group $G$.
One is equipped with a degree map $d$ assigning to each nonzero word in the letters of $\cala$
an element in $G$, such that
$d(vw) = d(v) d(w)$ and $d(v^*) = d(v)^{-1}$ when $w,v,wv \neq 0$.
The $*$-algebra $X$ together with these data $\cala,d$ and $G$ is called a balance system.

Let us anticipate that we shall apply this setting to a semigraph algebra $X$.
We define $G$ to be $\Z^k$, $\cala$ the standard words, and $d$ the degree map.

In \cite{burgiCancelling}, a criterion (C)$^*$ is given (explained below) which characterizes special
balance systems, which are then called cancelling systems.
If we have a cancelling system, and the word set is an inverse semigroup of partial
isometries then the cancelling system is even a so-called amenable cancelling system
(\cite{burgiCancelling}, Corollary 1).
Such a system satisfies the following uniqueness theorem
(\cite{burgiCancelling}, Theorem 2.1).

\begin{theorem}[\cite{burgiCancelling}, Theorem 2.1]   \label{theoremAmenableCancelling}
If $X$ is an amenable cancelling system then the universal $C^*$-representation $\pi:X \rightarrow C^*(X)$
(so $C^*(X)$ is the enveloping $C^*$-algebra)
is injective on the core, and actually this is the only existing $C^*$-representation
which is injective on the core (up to isomorphism).
\end{theorem}

If we can verify the condition (C)$^*$ for a semigraph algebra then it is a cancelling system.
It is then automatically an amenable system as the word set forms an inverse semigroup of partial isometries
(Lemma \ref{semigraphwordrep}).
Then the above theorem applies.

The criterion (C)$^*$ can now be formulated as follows:

There exists a subset $P$ of the core consisting of nonzero projections such that for any nonzero projection
$q$ in the core there is a projection $p$ in $P$ satisfying $p \precsim q$ (Murray--Von Neumann order).
There exists a subset $B$ of the algebra $X$ such that any word with nonzero degree can be expressed
as a linear combination of elements of $B$.
For every $x \in B$ and every $p \in P$ there is a $q \in P$ such that $q \le p$ and $q x q =0$.

We are going to introduce a definition which is designed to guarantee the validity of (C)$^*$.

\begin{definition}   \label{defCancel}
{\rm A semigraph algebra $X$ is called {\em cancelling} if for every
standard word $w$ with nonzero degree and every nonzero
standard projection $p$ there is a nonzero standard projection
$q$ such that $q \le p$ and $q w q = 0$.
}
\end{definition}

If $X$ is a cancelling semigraph algebra then it satisfies (C)$^*$. Indeed, define $B$
to be the standard words with nonzero degree, and $P$ the nonzero standard projections.
By Lemma \ref{semigraphwordrep} a word with nonzero degree may be expressed as a sum of words of $B$.
By Corollary \ref{corollarySemigraphMatrixunits}, any nonzero projection of the core
is larger or equal in Murray--Von Neumann order than a nonzero standard projection.
So (C)$^*$ is now evident.

Theorem \ref{theoremAmenableCancelling} thus yields the following Cuntz--Krieger uniqueness theorem.

\begin{theorem}[Cuntz--Krieger uniqueness theorem] \label{theoremamenablesemigraphsystem}
A cancelling semigraph algebra $X$ satisfies the following
uniqueness:

The universal representation $X \longrightarrow C^*(X)$ is injective on
the core, and so non-vanishing on the nonzero standard projections,
and up to
isomorphism this is the only existing representation of $X$ in a $C^*$-algebra which
is non-vanishing on nonzero standard projections and has dense image.
\end{theorem}

We used here also the fact that a representation is injective on the core if and only
if it is non-vanishing on nonzero standard projections (Corollary \ref{corollarySemigraphMatrixunits}).

\section{The quotient of a semigraph algebra}
\label{SectionSemigraphQuotient}

The following lemma tells us that a semigraph $X$ divided by a subset of the fiber space
(Definition \ref{defFiberSpaceX})
is a semigraph algebra again.

\begin{lemma}   \label{lemmaQuotientSemigraphAlgebra}
Let $X$ be a semigraph algebra and $Y$ the quotient of $X$ by a
subset of the fiber space of $X$.
Let $f: X \longrightarrow Y$ be the quotient map. Then $Y$ is a
semigraph algebra for the new generator sets
$\calp_{\rm new } = f(\calp)$ and $\calt_{\rm new} = f(\calt) \backslash \{0\}$.
The restriction
\begin{equation}  \label{restrictionQuotientMapF}
f|_{f^{-1}(\calt_{\rm new})}:{f^{-1}(\calt_{\rm new})} \longrightarrow \calt_{\rm new}
\end{equation}
is a bijection.
\end{lemma}

\begin{proof}
Since by Definition \ref{DefSemigraphAlgebra} (vi) $X$
is a quotient of $\F$ by a subset of the fiber space of $\F$,
and $Y$ is a quotient of a subset of a fiber space of $X$,
$Y$ may also be realised as a quotient of a subset of the fiber space of $\F$.
Hence, by Lemma \ref{lemmadegreemap}
$Y$ is
endowed with a degree map defined on the nonzero words of $Y$.
Since the gauge actions on $X$ and $Y$ are essentially identic, their degree maps are also essentially identic. 

In particular, $\calt_{\rm new}$ is endowed with a degree map
$d(f(t)) = d(t)$ for $t \in \calt$, $f(t) \neq 0$. To prove that (\ref{restrictionQuotientMapF}) is injective
(it is surely surjective),
suppose that $f(s) = f(t) \neq 0$ for $s,t \in
\calt$, and $s \neq t$. Then $s^* t = 0$ since $d(s)= d(f(s))=
d(f(t)) =d(t)$. Hence $f(s^* s)= f(s^* t)=0$, and so $f(s)= f(s s^*
s) = 0$, which is a contradiction.
Using this injectivity, it is now
easy to check that $(\calt_{\rm new})_1$, which is isomorphic to $(f^{-1}(\calt_{\rm new}))_1 \subseteq \calt_1$, is, as $\calt_1$, a
semigraph.

We are going to prove that $Y$ is a semigraph algebra. Definition \ref{DefSemigraphAlgebra} (vi) is verified
for $Y$. Definitions \ref{DefSemigraphAlgebra} (i)-(iv) are obvious.
It remains to check Definition \ref{DefSemigraphAlgebra} (v).
Suppose that $x,y,\alpha,\beta \in \calt_1$ and $f(x) f(\alpha)=f(y)f(\beta) \neq 0$.
Then $f(x \alpha) = f(y \beta) \neq0$. By injectivity of (\ref{restrictionQuotientMapF}),
$x \alpha = y \beta$.
Hence one has
\begin{equation} \label{formulaf}
(\calt_{\rm new})_1^{(\min)}(f(x),f(y)) =
\set{(f(\alpha),f(\beta))}{(\alpha,\beta) \in \calt_1^{(\min)}(x,y),
\,f(x) f(\alpha) \neq 0}
\end{equation}
for $x,y \in \calt_1$ with $f(x) \neq 0$ and $f(y)\neq 0$. In particular,
$(\calt_{\rm new})_1$ is finitely aligned. 
Applying the map $f$ to identity (\ref{defExpansion}) of Definition
\ref{DefSemigraphAlgebra} (v) we get
\begin{equation}  \label{formulaf2}
f(x)^* f(y) = \sum_{(\alpha,\beta) \in \calt_1^{(\min)}(x,y)}
f(\alpha) f(q_{x,y,\alpha,\beta}) f(\beta)^*.
\end{equation}
Since the left hand side of (\ref{formulaf2}) has the left unit $f(x)^* f(x)$, this must also be a left unit for the right hand side of (\ref{formulaf2}).
Imaging putting this unit before the sum in (\ref{formulaf2}), we see that the summands satisfying $f(x) f(\alpha) = 0$
vanish. So we drop these $\alpha$'s and end up with
$$f(x)^* f(y) = \sum_{(\alpha,\beta) \in \calt_1^{(\min)}(x,y),\, f(x) f(\alpha) \neq 0}
f(\alpha) f(q_{x,y,\alpha,\beta}) f(\beta)^*.$$
By (\ref{formulaf}) this verifies Definition
\ref{DefSemigraphAlgebra} (v) for $Y$.
%
%
\end{proof}

\section{Full semigraph algebras}
\label{SectionFullSemiAlgebra}

The aim of this section is to associate to a given semigraph algebra $X$ a further semigraph algebra $X_\|$
by adding relations to $X$ which are counterparts to the relation $s_1 s_1^* + s_2 s_2^*=1$
in the Cuntz algebra $\calo_2$.


\begin{definition}
{\rm
Write $\calt_1^{(\infty)}$ for the set of all increasing sequences in $\calt_1$. That means, an element $x \in \calt_1^{(\infty)}$
is a function $x:\N_0^k  \rightarrow \calt_1$ such that $d(x_n) = n$ and $x_{n_2}(0,n_1) = x_{n_1}$
for all $n_1 \le n_2$
($n_1,n_2 \in \N_0^k$).
}
\end{definition}

We may interpret an increasing sequence $x \in \calt_1^{\rm (min)}$ as an infinite path in $\calt_1$.

\begin{definition}    \label{definitionFull}
{\rm
Define $\calo$ to be the set of all standard projections $p \in X$
for which for every increasing sequence $x \in \calt_1^{(\infty)}$
one has $p x_n =0$ eventually for some $n$.
%

Then the {\em full semigraph
algebra} $X_\|$ associated to $X$ is the $*$-algebraic quotient of $X$ by $\calo$.
}
\end{definition}


Since $p x_n$ is a partial isometry with norm $1$, $\lim_n p x_n=0$ is equivalent to saying
that $p x_n =0$ eventually
(or to $p P_{x_n}= 0$ eventually).


The idea behind fullness is to add for every coordinate $i \in k$ the formal relation
$``1=\sum_{a \in \calt^{(e_i)}} P_a"$
to the semigraph algebra $X$. (This is Cuntz' relation in the Cuntz algebra \cite{cuntz} that the sum of the range projections of the generators is the unit.) This may however be an infinite sum, and so the meaning must be specified.
With these relations we get
$$``1 = \sum_{|a|= e_i} P_a = \sum_{|a|= e_i} \sum_{|b|= e_j} a b b^* a^*
= \ldots = \sum_{\alpha \in \calt^{(n)}} P_\alpha.$$
Thus an element $p$ in $X$ seems to vanishes if and only
if $p 1 = 0$ if and only if there is an $n \in \N^n_0$ such that $p P_\alpha=0$ for all $\alpha \in \calt_1^{(n)}$.
This condition is however somewhat too strong, and
so we heuristically think of the limits of the range projections
$P_{x_n}$ of elements $x \in \calt_1^{(\infty)}$ as the spectrum of a
commutative algebra generated by all range projections $P_\alpha$ ($\alpha \in \calt_1$).
Elements $x$ in the spectrum correspond to limits $x = \lim P_{x_n}$.
So we declare $p$ to be zero if the evaluation on the
spectrum is zero everywhere, that is, if $\lim_n p P_{x_n} = 0$ (equivalently $p x_n=0$ eventually)
for all $x
\in \calt_1^{(\infty)}$.
This is what we do in Definition \ref{definitionFull}.

In the next lemma
we shall show that the quotient of $X$ by $\calo$
is indeed a semigraph algebra and that it is indeed full in the sense that
we get nothing new if we consider the full semigraph algebra of this quotient again.

\begin{lemma} \label{lemmafullsystem}
\begin{itemize}
\item[(a)]
$X_\|$ is a semigraph algebra, and ${(X_\|)}_\| = X_\| \,.$

\item[(b)]
If $p$ is a standard projection in $X$ then $p$ vanishes in $X_\|$ if and only if
$p \in \calo$.

\item[(c)]
If $t
\in \calt$ then $t$ vanishes in $X_\|$ if and only if $P_{t} \in \calo$.
\end{itemize}
\end{lemma}

\begin{proof}
$X_\|$ is a semigraph algebra as it is the quotient of the semigraph algebra $X$ by a subset of the core, which is in the fiber space (Lemma \ref{lemmaQuotientSemigraphAlgebra}). We denote the
equivalence class of $x \in X$ in $X_\|=X /\langle \calo \rangle$ by $[x]$.
We are going to prove (b). Suppose that $p$
is a standard projection and $[p]= 0$.
Then there are elements $p_i \in \calo$, scalars $\alpha_i \in \C$,
and words $v_i, w_i$ such that
\begin{equation}   \label{psum}
p = \sum_{i=1}^\kappa \alpha_i v_i p_i w_i.
\end{equation}
Since by Lemma
\ref{semigraphwordrep} every word may be written as a sum of standard words,
we may assume that the $v_i$ and $w_i$'s are standard words.
Say that $w_i = s_i q_i t_i^*$ for $s_i,t_i \in \calt_1$ and $q_i \in \calp$.
Let $x \in \calt_1^{(\infty)}$.
If $d(x_n) \ge d(t_i)$ then either $t_i^* x_n = 0$, or $t_i^* x_n \neq 0$ in which case $x_n(0,d(t_i))= t_i$.
Hence, $w_i x_n = 0$ eventually for some $n$, or
$$p_i w_i x_n = p_i s_i q_i t_i^*  x_n = p_i s_i q_i x(d(t_i),n) = p_i s_i x(d(t_i),n) q_{i,n},$$
which is also vanishing eventually for some $n$ as $p_i \in \calo$
(here $q_i$ skips $x(d(t_i),n)$ and becomes $q_{i,n} \in \calp$ by Definition \ref{DefSemigraphAlgebra} (iv)).
Hence, by (\ref{psum}), $p x_n = 0$ eventually. Since $x$ was arbitrary, $p \in \calo$ by Definition
\ref{definitionFull}. 

We are going to show that ${(X_\|)}_\| = X_\|$.
To this end we need to show that $\calo_{\|}$ (i.e. $\calo$ with respect to $X_\|$)
is $\{0\}$.
Let $[p]$ be a standard projection in $X_\|$ ($p$ denoting a standard projection in $X$).
Suppose that $[p]$ is in $\calo_\|$. 
Then by Definition \ref{definitionFull} for every $x \in \calt_1^{(\infty)}$ $[p][x_n]=0$ eventually for some $n$.
For simplicity let us assume that $p=a q a^* (1-b b^*)$ for some $a,b \in \calt_1$ and $q \in \calp$.
Let $x \in \calt_1^{(\infty)}$.
If $d(x_n) \ge d(a)$ then $x_n x_n^* a q a^* = 0$, or $x_n x_n^* a q a^* = x_n q' x_n^*$ for some $q' \in \calp$
satisfying $q x_n(0,|a|) = x_n (0,|a|) q'$.
Hence
$$p x_n x_n^* = x_n q' x_n^* (1-b b^*)$$
is a standard projection for all $n \ge d(a)$.
Thus, since also $[p x_n x_n^*] =0$ for almost all $n$ (as $[p] \in \calo_\|$),
by Lemma \ref{lemmafullsystem} (b), which we have proved, $p x_n x_n^* \in \calo$ for almost all $n$.
Fix any such an $n$.
Then, $p x_n x_n^* x_m x_m^* = 0 =p x_m x_m^*$ for almost $m \ge n$.
Since $x$ was arbitrary, $p \in \calo$. Thus $[p] = 0$.


(c) follows from $[t] = 0$ if and only
if $[t t^*] = 0$ if and only if $t t^* \in \calo$ by (b).
\end{proof}

\begin{definition}   \label{definitionFullness}
{\rm
A semigraph algebra $X$ is called {\em full} if $X= X_{\|}$.
}
\end{definition}

We shall introduce a condition for a semigraph algebra called aperiodicity which implies that the
semigraph algebra is cancelling when it is also full.
The aperiodicity condition is more or less a condition directly for the underlying semigraph.

\begin{definition}  \label{definitionAperiodicity}
{\rm A semigraph algebra $X$ is called {\em aperiodic} if for
every {\em elementary} standard projection $e$, every $x \in \calt$ with $x e \neq 0$, and all
distinct $0 \le m,n \le d(x)$, there exists a $y$ in $\calt_1$ such that $x e y \neq 0$ and
\begin{equation}  \label{mce0aperiodic}
\calt_1^{(\min)}\Big ((xy) (m,d(xy)),(xy) (n,d(xy)) \Big) =
\emptyset.
\end{equation}
}
\end{definition}



\begin{proposition}  \label{fullaperiodicsystemamenable}
An aperiodic full semigraph algebra is cancelling.
\end{proposition}

\begin{proof}
We will check that $X$ is cancelling (Definition \ref{defCancel}). Let $w$ be a standard word with
nonzero degree and $p$ a nonzero standard projection. We need to find a nonzero standard projection
$q$ such that $q \le p$ and $qwq=0$.
We may
write $w$ as $w = \alpha Q \beta$ for $Q \in \calp$ and $\alpha,\beta \in
\calt_1$ with $|\alpha| \neq |\beta|$. We may write
\begin{equation}   \label{seq56}
p= t_0 q_0 t_0^* (1- t_1 q_1 t_1^*) \ldots (1- t_n q_n t_n^*)
\end{equation}
for certain $t_i \in \calt_1$ and $q_i \in \calp$. Since $X$ is full, $p \notin \calo$. Thus there is an $x \in \calt_1^{(\infty)}$ such that $p x_i \neq 0$ for all $i \in \N_0^k$.
Fix any $N >
\max(|\alpha|,|\beta|,|t_0|,\ldots,|t_n|)$.
Then $p x_N x_N^* \neq 0$.
Note that for every $0 \le i \le n$, either $t_i q_i t_i^* x_N x_N^* = 0$ or
$t_i =x_N(0,|t_i|)$, in which case 
$$t_i q_i t_i^* x_N x_N^* = x_N q_i' x_N^*$$
for some $q_i' \in \calp$ by Definition \ref{DefSemigraphAlgebra} (iv) and (v).
Thus $p x_N x_N^*$ is something like
\begin{eqnarray}
&& x_N q_0' x_N^* (x_N x_N^* - x_N q_1' x_N^*) \ldots (x_N x_N^* - x_N q_n' x_N^*)  \nonumber\\
&=& x_N q_0' x_N^* x_N  (1 - q_1')  x_N^* x_N \ldots x_N^* x_N (1 - q_n') x_N^*  \nonumber \\
&=& x_N \big( q_0'  (1 -  q_1') \ldots (1 -  q_n') \big) x_N^*   \label{middleelem}\\
&=& x_N e x_N^*,   \label{middleelem2}
\end{eqnarray}
where $e$ denotes the elementary standard projection appearing in the middle of (\ref{middleelem}).

Since $x_N e \neq 0$, we may choose a $y
\in \calt_1$ by the aperiodicity condition such that $x_N e y \neq 0$ and
\begin{equation}   \label{minmiddlef}
\calt_1^{(\min)}(z(|\alpha|,|z|), z(|\beta|,|z|)) = \emptyset
\end{equation}
for $z= x_N y$.
Thus $0 \neq q:= P_{x_N e y} \le P_{x_N e} = p x_N x_N^* \le p$.
We may write
$$q= x_N e y y^* x_N^* = x_N y e' y^* x_N^* = z e' z^*$$
for some elementary standard projection $e'$ satisfying $e y= y e'$ by successive
application of Definition \ref{DefSemigraphAlgebra} (iv).
We then have
\begin{eqnarray*}
q w q &=&
z e' z^* \alpha Q \beta^* z e' z^*\\
&=&  z e' z(|\alpha|,|z|)^* Q_\alpha Q Q_\beta  z(|\beta|,|z|) e' z^*\\
&=& z e' z(|\alpha|,|z|)^*  z(|\beta|,|z|)  Q_\alpha' Q' Q_\beta'  e' z^*\\
&=& 0
\end{eqnarray*}
by (\ref{minmiddlef}), provided that $z^* \alpha \neq 0$ and $\beta^* z \neq 0$
(if not so, we obviously obtain zero anyway).
%
\end{proof}

\begin{lemma}
A representation of a full semigraph algebra is injective on the
core if and only if it is non-vanishing on elementary standard projections.
\end{lemma}

\begin{proof}
Let $\pi$ be a representation which is non-vanishing on nonzero elementary standard projections.
By Corollary \ref{corollarySemigraphMatrixunits} we must show that
$\pi$ is non-vanishing on every nonzero standard projection $p$.
Assume that $\pi(p)=0$. We go into the proof
of Proposition \ref{fullaperiodicsystemamenable} again, and assume
(\ref{seq56}). Again, by fullness we have $p x_N x_N^* \neq 0$ for a
certain $x_N \in \calt$.
Then $p x_N x_N^* = x_N e x_N^*$, see (\ref{middleelem2}), and thus
$x_N^* p x_N=x_N^* ( p x_N x_N^*) x_N = x_N^* x_N e$ is a nonzero elementary standard projection.
Since $\pi(p)=0$, $0 = \pi(x_N^* p x_N) = \pi(x_N^* x_N e)$,
which contradicts the assumption that $\pi$ is non-vanishing on nonzero elementary
standard projections.
%
\end{proof}


\bibliographystyle{plain}
\bibliography{references}

\end{document}